\documentclass[a4paper,11pt,oneside]{article}
\usepackage{blindtext}
\usepackage{xcolor}
\usepackage[latin1]{inputenc}
\usepackage[russian,english]{babel}
\usepackage{amsmath,amssymb}
\usepackage{textcomp}
\usepackage{bbm}
\usepackage[T2A,T1]{fontenc}
\usepackage{amsthm}
\usepackage[all,cmtip]{xy}
\usepackage{mathtools}
\usepackage{tikz}
\usepackage{tikz-cd}
\usetikzlibrary{matrix}
\usepackage{tocbibind}
\usepackage{graphicx}
\usepackage{wrapfig}
\usepackage{mathrsfs}
\usepackage{fancyhdr}
\usepackage[colorinlistoftodos]{todonotes}
\usepackage{comment}
\usetikzlibrary{arrows,decorations.pathmorphing}
\title{}
\author{}
\date{}
\usetikzlibrary{calc,trees,positioning,arrows,chains,shapes.geometric,%
	decorations.pathreplacing,decorations.pathmorphing,shapes,%
	matrix,shapes.symbols}
\usepackage{titlesec}
\usepackage{graphicx}

\titleformat{\chapter}[display]
{\bfseries\huge}
{\filcenter\MakeUppercase{\chaptertitlename} \Huge\thechapter}
{1ex}
{\titlerule\vspace{1ex}\filcenter}
[\vspace{1ex}\titlerule]
\usepackage{scrextend}

\newtheorem{thm}{Theorem}

\newtheorem{prop}{Proposition}

\newtheorem{coroll}{Corollary}

\newtheorem{lemma}{Lemma}

\newcommand{\Z}{\mathbb{Z}}
\newcommand{\Co}{\mathbb{C}}

\newcommand{\Q}{\mathbb{Q}}
\newcommand{\p}{\mathbb{P}}
\newcommand{\Qal}{\overline{\Q}}

\theoremstyle{plain} 
\newcommand{\thistheoremname}{}
\newtheorem{genericthm}[thm]{\thistheoremname}

\newtheorem*{genericthm*}{\thistheoremname}
\newenvironment{namedthm*}[1]
{\renewcommand{\thistheoremname}{#1}%
	\begin{genericthm*}}
	{\end{genericthm*}}

\DeclareMathOperator{\ai}{a}
\DeclareMathOperator{\bi}{b}
\DeclareMathOperator{\ci}{c}

\DeclareMathOperator{\Frob}{Frob}

\DeclareMathOperator{\paruno}{(}
\DeclareMathOperator{\pardue}{)}

\DeclareMathOperator{\lcm}{lcm}

\makeatletter
\newcommand*{\math@version@bold}{bold}
\DeclareMathOperator\Sha{
	\textrm{%
		\usefont{T2A}{cmr}{\ifx\math@version\math@version@bold bx\else m\fi}{n}%
		\CYRSH
	}%
}
\makeatother

\begin{document}
	\begin{center}
		\textbf{AN AVERAGE VERSION OF CILLERUELO'S CONJECTURE FOR FAMILIES OF $S_n$-POLYNOMIALS OVER A NUMBER FIELD }
	\end{center}
\begin{center}
	Ilaria Viglino
\end{center}
\begin{center}
	\textbf{Abstract}
\end{center}

\begin{addmargin}[2em]{2em}

\fontsize{10pt}{12pt}\selectfont
	For $ f\in\Z[X] $ an irreducible polynomial of degree $ n $, the Cilleruelo's conjecture states that$$\log(\lcm(f(1),\dots,f(M)))\sim(n-1)M\log M$$as $ M\rightarrow+\infty $, where $ \lcm(f(1),\dots,f(M)) $ is the least common multiple of $f(1),\dots,f(M)$. It's well-known for $ n=1 $ as a consequence of Dirichlet's Theorem for primes in arithmetic progression, and it was proved by Cilleruelo in \cite{Cil} for quadratic polynomials. Recently the conjecture was shown by Rudnick and Zehavi in \cite{RZ} for a large family of polynomials of any degree. We want to investigate an average version of the conjecture for $S_n$-polynomials with integral coefficients over a fixed extension $K/\Q$ by considering the least common multiple of ideals of $\mathcal{O}_K$.
\end{addmargin}
\normalsize

\section{Introduction}

We fix a field extension $K/\Q$ of degree $d$. Let $n\ge2$ and let $N$ be positive integers. We consider monic polynomials with coefficients in $\mathcal{O}_K$ of the form$$f(X)=X^n+\alpha_{n-1}X^{n-1}+\dots+\alpha_0.$$Choose an ordered integral basis $(\omega_1,\dots,\omega_d)$ of $\mathcal{O}_K$ over $\Z$. We have, for all $k=0,\dots,n-1$,$$\alpha_k=\sum_{i=1}^{d}a_{i}^{(k)}\omega_i$$for unique $a_{i}^{(k)}\in\Z$. We view the coefficients $a_{i}^{(k)}$ as independent, identically distributed random variables taking values uniformly in $\{-N,\dots,N\}$. Define the \textit{height} of $\alpha_k$ as $\mbox{ht}(\alpha_k)=\max_{i}|a_{i}^{(k)}|$ and the \textbf{height} of the polynomial $f$ to be$$\mbox{ht}(f)=\max_{k}\mbox{ht}(\alpha_{k}).$$For $n\ge2$, $N>0$ define$$\mathscr{P}_{n,N}^{0}(K)=\{f\in\mathcal{O}_K[X]:\mbox{ht}(f)\le N,\ G_{K_f/K}\cong S_n\},$$where $K_f$ is the splitting field of $f$ over $K$ inside a fixed algebraic closure $ \Qal $ of $ \Q $. We call these polynomials \textbf{$S_n$-polynomials over $K$}, or simply $S_n$-polynomials when there is no need to specify the base field.\\

\subsection{Main result}
	If $\lambda_1,\dots,\lambda_s$ are elements of $\mathcal{O}_K$, we can factorize the ideals they generate in the Dedekind domain $\mathcal{O}_K$ as$$\lambda_i\mathcal{O}_K=\prod_{\wp\subseteq\mathcal{O}_K}\wp^{\beta_{\wp}^i}$$for all $i$, where $\beta_{\wp}^i=0$ for all but finitely many $\wp$. The \textit{least common multiple} of $\lambda_1,\dots,\lambda_s$ is the ideal of $\mathcal{O}_K$ defined as the leatest common multiple of the ideals $\lambda_1\mathcal{O}_K,\dots,\lambda_s\mathcal{O}_K$ in the Dedekind domain $\mathcal{O}_K$, that is,$$\mbox{lcm}(\lambda_1,\dots,\lambda_s)=\underset{\wp\subseteq\mathcal{O}_K}{\bigcap}\wp^{\max\{\beta_{\wp}^1,\dots,\beta_{\wp}^s\}}=\underset{\wp\subseteq\mathcal{O}_K}{\prod}\wp^{\max\{\beta_{\wp}^1,\dots,\beta_{\wp}^s\}}.$$Let $n\ge3$. It is known that the polynomials in $\mathscr{P}_{n,N}^{0}(K)$ are generically irreducible.
	\begin{thm}\label{thm2}
			Let $N,M>0$ such that$$M(\log M)^\ell\ll N=o\left( M\frac{\log M}{\log\log M}\right) $$for some $0<\ell<1$. Then \begin{multline*}
			\log|N_{K/\Q}(\lcm(f(\lambda):\lambda\in\mathcal{O}_K,\ N_{K/\Q}\lambda\le M)|)=(n-1)M\log M\\
			+O\left( M\frac{\log M}{\log\log M}+N\log\log M\right),
		\end{multline*}asymptotically almost surely when $N,M\rightarrow+\infty$.
	\end{thm}
This result will be proved in Section 4.2. The case of linear polynomials is dealt with separately in Section 4.1, as a consequence of Dirichlet's Theorem for number fields. For degree-2 polynomials, it is possible to obtain explicit asymptotics for the least common multiple, analogously to the ones in \cite{Cil} for polynomials in $\Z[X]$. However, the latter is not a subject of the current paper. \\
Recently Cilleruelo's conjecture was shown for a large family of integer polynomials of any degree by Rudnick and Zehavi in \cite{RZ}. In particular, when restricted to the case $K=\Q$, our theorem also encompasses non-included cases in the above-mentioned result.

\section{Counting $S_n$-polynomials over $K$}
It has been proven that almost all polynomials are $S_n$-polynomials in the following sense:$$\frac{|\mathscr{P}_{n,N}^{0}(K)|}{|\mathscr{P}_{n,N}(K)|}\underset{N\rightarrow+\infty}{\longrightarrow}1.$$For instance, in the case $K=\Q$, Van der Waerden gave in \cite{Wa} an explicit error term $ O(N^{n-\frac{1}{6(n-2)\log\log N}}) $. It has improved in \cite{Gal} using large sieve to $ O(N^{-1/2}\log N) $, and more recently by Dietmann \cite{Di} using resolvent polynomials to $O(N^{-2+\sqrt{2}+\varepsilon})$. The best estimate can be found in \cite{Bh1}, who proved the following result, conjectured by van der Waerden.
\begin{namedthm*}{Theorem}[Bhargava]
	If $n\ge5$, one has, $$|\mathscr{P}_{n,N}^{0}(\Q)|=(2N)^n+O(N^{n-1}),$$as $ N\rightarrow\infty $.
\end{namedthm*}The cubic and quartic cases of van der Waerden's conjecture were proved by Chow and Dietmann in \cite{CD}.\\

In his work, Bhargava used a combination of algebraic techniques and Fourier analysis over finite fields. In the below theorem, we generalize this result for polynomials in $\mathscr{P}_{n,N}^{0}(K)$, for certain values of $n$ and $d$. We also use large sieve over number fields to prove the upper bound for all $n\ge3$ and $d\ge1$.
\begin{thm}\label{thm1}
	Let $d\ge1$ and $n\ge2$. There exist positive constants $\theta$ and $\theta_n$ such that the number of non $S_n$-polynomials is$$|\mathscr{P}_{n,N}(K)\setminus\mathscr{P}_{n,N}^{0}(K)|\ll_{n,K}N^{d(n-\theta)}(\log N)^{\theta_n},$$as $N\rightarrow+\infty$. In particular,\begin{enumerate}
		\item[$\paruno \ai\pardue$] if $n=2$, we can choose $\theta=1$, $\theta_2=1$;
		\item[$\paruno \bi\pardue$] for all $d\ge1$ and $n\ge3$ the above estimate holds with $\theta=1/2$ and $\theta_n=1-\gamma_n$, where $\gamma_n\sim(2\pi n)^{-1/2}$;
		\item[$\paruno \ci\pardue$] if one of the following conditions is satisfied, we can take $\theta=1$ and $\theta_n=0$:
		\begin{enumerate}
			\item[$\bullet$] $d=1$, $n=3,4$;
			\item[$\bullet$] $\left[ \frac{2d+\sqrt{4d^2-2d}}{d}\right]+1\le n\le 5$;
			\item[$\bullet$] $d\le 23$, $2(2d+1)\le n\le 94$;
			\item[$\bullet$] $n\ge\max(95,2(2d+1))$.
		\end{enumerate}
	\end{enumerate}
\end{thm}

The first bullet of (c) of Theorem \ref{thm1} is van der Waerden's conjecture for cubic and quartic fields, which has already been proven by Chow and Dietmann in \cite{CD}. The other bullets of (c) of Theorem \ref{thm1} are a generalization of this result for polynomials with integral coefficients in a number field $K$, for some values of $d$ and $n$.
The reason why we split different values of $n$ and $d$ is that in order to estimate the number of polynomials having primitive Galois group, we need an upper bound on the number of field extensions of fixed degree, with bounded discriminant. Different asymptotic formulas as known according to the degree. Finally, for (b) we apply large sieve to the set $\mathscr{P}_{n,N}$ (see \cite{Gal} for the analogous result for $d=1$). A complete proof of this result can be found in \cite{Vi}.

\section{Prime splitting densities}

From now on, according to Theorem \ref{thm1}, we set $\xi>0$ so that the number of non $S_n$-polynomials in $\mathscr{P}_{n,N}(K)$ is $\ll_{n,K}N^{d(n-\xi)}$. Specifically, for all $n\ge3,\ d\ge1$ we can take $\xi=\frac{1}{2}-\varepsilon$ for an $\varepsilon>0$ arbitrary small. If moreover $n$ is as in (c) of Theorem \ref{thm1}, put $\xi=1$.
	
For every splitting type $ r $, and every prime $\wp$ of norm $q_\wp$, recall that we denoted by $X_{n,r,\wp}$ the set of polynomials in $\mathbb{F}_{q_\wp}$ with square-free factorization of type $r$.
The following key fact is what we'll use to estimate the error term in the asymptotic of the expectation $\mathbb{E}_N(\pi_{f,r}(x))$ of $\pi_{f,r}(x)$ and its powers.
\begin{lemma}\label{l11}
	Let $k\ge 1$, $\wp_1,\dots, \wp_k$ primes and $g_i\in X_{n,r,\wp_i}$ for all $i=1,\dots,k$. Then if $q_{\wp_i}<N^{d\xi/kn}$ for all $i=1,\dots,k$ $$\p_N(f\in\mathscr{P}_{n,N}^{0}:f\equiv g_i\mod \wp_i\ \ \forall i=1,\dots,k)=\frac{1}{(q_{\wp_1}\dots q_{\wp_k})^n}+O_{n,K}(N^{-d\xi})$$as $N\rightarrow+\infty$.

\end{lemma}

\begin{proof}We prove the case $k=1$. An application of the Chinese Remainder Theorem leads to the result for $k>1$.\\
	Let $ g=\sum_{i=1}^{n}g_iX^i $ and $ f=\sum_{i=1}^{n}f_iX^i $. Now $(\omega_1,\dots,\omega_d)$ is an integral basis of $\mathcal{O}_K$ over $\Z$; by applying linear transformations we can assume that the reduction modulo $\wp$ of $(\omega_1,\dots,\omega_{f_\wp})$ is a basis for the $\mathbb{F}_p$-vector space $\mathcal{O}_K/\wp$. Then write for every $i=0,\dots,n-1$\begin{align*}
		g_i&=\sum_{j=1}^{f_\wp}b_j^{(i)}\omega_j\mod \mathbb{F}_{q_\wp},\\
		f_i&=\sum_{j=1}^{f_\wp}a_j^{(i)}\omega_j\mod \mathbb{F}_{q_\wp},
	\end{align*}where $a_j^{(i)},b_j^{(i)}\in\Z$ for all $i$ and $j$.
	
	One has $ f\equiv g\mod \wp $ if and only if $ f_i= g_i $ in $\mathbb{F}_{q_\wp}$ for $ i=0,\dots,n-1 $.
	This means $ a_j^{(i)}\equiv b_j^{(i)} \mod p$, that is $a_j^{(i)}= b_j^{(i)}+p k_j^{(i)}$ for some $ k_j^{(i)}\in\Z $. Since the height of $f$ is less or equal than $N$, for $j=1,\dots,f_\wp$ and for all $i=0,\dots,n-1$ we have$$\frac{-N-b_j^{(i)}}{p}\le k_j^{(i)}\le\frac{N-b_j^{(i)}}{p};$$so for each of the coefficients $a_{1}^{(i)},\dots,a_{f_\wp}^{(i)}$ we have $$\left[ \frac{N-b_j^{(i)}}{p}\right] -\left[\frac{-N-b_j^{(i)}}{p}\right]=\frac{2N}{p}+O(1)$$choices. Whereas for each coefficient $a_{f_\wp+1}^{(i)},\dots,a_{d}^{(i)}$ there are $2N$ choices. Therefore for each coefficient $f_i$ of $f$ one has$$\Big(\frac{2N}{p}+O(1)\Big)^{f_\wp}\cdot(2N)^{d-f_\wp}=\frac{(2N)^d}{q_\wp}+O(N^{d-1})$$possibilities. It turns out that
	$$|\{f\in\mathscr{P}_{n,N}:f\equiv g\mod \wp\}|=\frac{(2N)^{nd}}{q_\wp^n}+O(N^{dn-1}),$$so by Theorem \ref{thm1}\begin{align*}
		|\{f\in\mathscr{P}_{n,N}^0:f\equiv g\mod \wp\}|&=\underset{f\equiv g\tiny\mbox{ mod }\wp}{\sum_{f\in\mathscr{P}_{n,N}}}1+O\Big(\sum_{f\notin\mathscr{P}_{n,N}^{0}}1\Big)\\&=\frac{(2N)^{nd}}{q_\wp^n}+O(N^{d(n-\xi)}).
	\end{align*}
	As long as $ q_\wp^n<N^{d\xi} $, we get\begin{align*}
		\frac{1}{|\mathscr{P}_{n,N}^{0}|}\underset{f\equiv g\tiny\mbox{ mod }\wp}{\sum_{f\in\mathscr{P}_{n,N}^{0}}}1&=\frac{1}{(2N)^{nd}}(1+O(N^{d(n-\xi)}))\Big(\frac{(2N)^{nd}}{q_\wp^n}+O(N^{d(n-\xi)})\Big)\\
		&=(1+O(N^{-d\xi}))\Big(\frac{1}{q_\wp^n}+O(N^{-d\xi})\Big)\\
		&=\frac{1}{q_\wp^n}+O(N^{-d\xi}).
	\end{align*}
\end{proof}

\begin{prop}\label{p9} One has, for all primes $\wp$ with $ q_\wp<N^{d\xi/(n+1)} $,
	\begin{enumerate}
		\item[$\paruno\ai\pardue$] $ \p_N(\mathbbm{1}_{f,r}(\wp)=1)=\mathbb{E}_N(\mathbbm{1}_{f,r}(\wp))=\delta(r)+\frac{C_{r}}{q_\wp}+O\Big(\frac{1}{q_\wp^2}+q_\wp^nN^{-d\xi}\Big) $,\\for some explicit constant $ C_{r}$;
		\item[$\paruno\bi\pardue$] $ \sigma^2_N(\mathbbm{1}_{f,r}(\wp))=(\delta(r)-\delta(r)^2)+\frac{C_{r}(1-2\delta(r))}{q_\wp}+O\Big(\frac{1}{q_\wp^2}+q_\wp^nN^{-d\xi}\Big) $. 
	\end{enumerate}It follows that, for $ x<N^{d\xi/(n+1)} $,
	\begin{enumerate}
		\item[$\paruno\ci\pardue$] $ \mathbb{E}_N(\pi_{f,r}(x))=\delta(r)\pi_K(x)+C_{r}\log\log x+O_{n,K}(1),$
	\end{enumerate}as $x,N\rightarrow+\infty$.
\end{prop}
Hence, the \textit{normal order} of $ \pi_{f,r}(x) $ is $ \delta(r)\pi_K(x) $, which means that $ \pi_{f,r}(x)\sim \delta(r)\pi_K(x)$ for almost all $ f $, as $ x\rightarrow+\infty $ and $ N $ large enough. 
\begin{proof}
	
	Once fixed a prime $ \wp $,
	\begin{align*}
		\mathbb{E}_N(\mathbbm{1}_{f,r}(\wp))&=\frac{1}{|\mathscr{P}_{n,N}^{0}|}\sum_{f\in\mathscr{P}_{n,N}^{0}}\mathbbm{1}_{f,r}(\wp)\\
		&=\frac{1}{|\mathscr{P}_{n,N}^{0}|}\underset{f\tiny{\mbox{ of splitting type }}r\tiny{\mbox{ mod }} \wp}{\sum_{f\in\mathscr{P}_{n,N}^{0}}}1\\
		&=\frac{1}{|\mathscr{P}_{n,N}^{0}|}\sum_{g\in X_{n,r,\wp}}\underset{f\equiv g\tiny\mbox{ mod }\wp}{\sum_{f\in\mathscr{P}_{n,N}^{0}}}1.
	\end{align*}

	On the other hand,$$|X_{n,r,\wp}|=\prod_{k=1}^{n}\binom{A_{q_\wp,k}}{r_k},$$where $ A_{q_\wp,k} $ is the number of degree-$ k $ irreducible polynomials in $ \mathbb{F}_{q_\wp}[X] $, which, by the M\"{o}bius inversion formula, equals$$\frac{1}{k}\sum_{d|k}\mu(d)q_{\wp}^{k/d}=\frac{q_\wp^k}{k}+O(q_\wp^{\alpha_k}),$$where $\alpha_k=1$ if $k=2$, and $\alpha_k< k-1$ if $k>2$. One has, for all $ k\ge2$\begin{align*}
		\binom{A_{q_\wp,k}}{r_k}&=\frac{A_{q_\wp,k}(A_{q_\wp,k}-1)\dots(A_{q_\wp,k}-r_k+1)}{r_k!}\\
		&=\frac{1}{r_k!}\Big(\frac{q_\wp^k}{k}+O(q_\wp^{\alpha_k})\Big)\dots\Big(\frac{q_\wp^k}{k}-r_k+1+O(q_\wp^{\alpha_k})\Big).
	\end{align*}It turns out that$$
	\binom{A_{q_\wp,k}}{r_k}=\begin{cases}
		\frac{1}{r_1!}q_\wp(q_\wp-1)\dots(q_\wp-r_1+1)&\mbox{ if }k=1\\
		\frac{1}{r_2!2^{r_2}}q_\wp^{2r_2}+C(r_2)q_\wp^{2r_2-1}+O(q_\wp^{2r_2-2})&\mbox{ if }k=2\\
		\frac{1}{r_k!k^{r_k}}q_\wp^{kr_k}+O(q_\wp^{k(r_k-1)+\alpha_k})&\mbox{ if }k>1.
	\end{cases}
	$$Hence\begin{multline*}
		|X_{n,r,\wp}|=\frac{1}{r_1!}q_\wp(q_\wp-1)\dots(q_\wp-r_1+1)\frac{1}{r_2!2^{r_2}}(q_\wp^{2r_2}+C(r_2)q_\wp^{2r_2-1}+O(q_\wp^{2r_2-2}))\\\prod_{k=3}^{n}(\frac{1}{r_k!k^{r_k}}q_\wp^{kr_k}+O(q_\wp^{k(r_k-1)+\alpha_k}))\\
		=\delta(r)q_\wp^{n}+C_rq_\wp^{n-1}+O(q_\wp^{n-2}),
	\end{multline*}where $ C_r=-\delta(r)C(r_2)\frac{(r_1+1)(r_1+2)}{2r_1!}.$
	
	By Lemma \ref{l11}, for $ q_\wp^{n+1}<N^{d\xi}$,\begin{align*}
		\mathbb{E}_N(\mathbbm{1}_{f,r}(\wp))&=(\delta(r)q_\wp^n+C_{r}q_\wp^{n-1}+O(q_\wp^{n-2}))\Big(\frac{1}{q_\wp^n}+O(N^{-d\xi})\Big)\\
		&=\delta(r)+\frac{C_{r}}{q_\wp}+O\Big(\frac{1}{q_\wp^2}+q_\wp^nN^{-d\xi}\Big),
	\end{align*}which proves (a) and (b) follows by definition.
	
	For (c), by linearity, we simply have to sum over all primes $\wp$ with $N_{K/\Q}\wp \le x $ and use the estimate$$\sum_{N_{K/\Q}\wp\le x}\frac{1}{N_{K/\Q}\wp}=\log\log x+O(1)$$to get$$\mathbb{E}_N(\pi_{f,r}(x))=\delta(r)\pi_K(x)+C_{r}\log\log x+O(1+\pi_K(x)^{n+1}N^{-d\xi})$$as long as$$\pi_K(x)^{n+1}N^{-d\xi}=o(\log\log x).$$If moreover $ x<N^{d\xi/(n+1)} $, then the term $ \pi_K(x)^{n+1}N^{-d\xi} $ is negligible.
\end{proof}

	\section{The Cilleruelo's conjecture for integral polynomials over $K$}

\subsection{Linear polynomials}	

For $ f\in\Z[X] $ an irreducible polynomial of degree $ n $, the Cilleruelo's conjecture states$$\log(\lcm(f(1),\dots,f(M)))\sim(n-1)M\log M$$as $ M\rightarrow+\infty $, where $ \lcm(f(1),\dots,f(M)) $ is the least common multiple of $f(1),\dots,f(M)$. It's well-known for $ n=1 $ by exploiting Dirichlet's Theorem for primes in arithmetic progression (see \cite{BKS} for a proof), to get an asymptotic estimate when $f(X)=Xk+h$, $(h,k)=1$:$$\psi_f(N)\sim N\frac{k}{\varphi(k)}\underset{(n,k)=1}{\sum_{1\le n\le k}}\frac{1}{n},$$as $N\rightarrow+\infty$.\\

We recall the definition of ray class group and we fix some notations. A \textbf{modulus} $\mathfrak{m}=\prod_{\wp}\wp^{\mathfrak{m}(\wp)}$ on $K$ is a function $\mathfrak{m}:\{\mbox{primes of }K\}\rightarrow\Z$ such that $$\begin{cases}
	\mathfrak{m}(\wp)\ge0\mbox{ for all }\wp,\\
	\mathfrak{m}(\wp)=0\mbox{ for all but finitely many }\wp,\\
	\mathfrak{m}(\wp)=0\mbox{ or }1\mbox{ if }\wp\mbox{ is real},\\
	\mathfrak{m}(\wp)=0\mbox{ if }\wp\mbox{ is complex}.
\end{cases}$$Let
$K_{\mathfrak{m},1}=\{a\in K^{\times}:\mbox{ord}_\wp(a-1)\ge\mathfrak{m}(\wp)\mbox{ for }\wp|\mathfrak{m}\mbox{ finite},\ a_\wp>0\mbox{ for }\wp|\mathfrak{m}\mbox{ real}\}$. The \textbf{ray class group modulo} $\mathfrak{m}$ is$$C_{\mathfrak{m}}=I^{S(\mathfrak{m})}/K_{\mathfrak{m},1},$$where $I^{S(\mathfrak{m})}$ is the free abelian group generated by the prime ideals not in $\mathfrak{m}$.\\
Let $\mathfrak{c}\in I^{S(\mathfrak{m})}$; the corresponding Chebyshev's function is$$\theta(x;\mathfrak{c},\mathfrak{m})=\underset{\wp\sim\mathfrak{c}\tiny\mbox{ in }C_{\mathfrak{m}}}{\sum_{N_{K/\Q}\wp\le x}}\log (N_{K/\Q}\wp)=\frac{x}{h_{\mathfrak{m}}}+o(x),$$as $x\rightarrow+\infty$, where $h_{\mathfrak{m}}$ is the \text{ray class number}, i.e. $h_{\mathfrak{m}}=|C_{\mathfrak{m}}|$.\\
This is equivalent to saying that the set of prime ideals $T$ of $K$ congruent to $\mathfrak{c}$ in $C_{\mathfrak{m}}$ has Dirichlet density $1/h_{\mathfrak{m}}$:$$\sum_{\wp\in T}\frac{1}{N_{K/\Q}\wp}\sim\frac{1}{h_{\mathfrak{m}}}\log\left(\frac{1}{s-1} \right),$$as $s\downarrow 1$.\\
To achieve explicit error terms, we want effective versions of the above asymptotic. We state here both a conditional and unconditional results proved by Lagarias and Odlyzko in 1977 \cite{LO}.
	\begin{thm}\label{thm3}
	Assume the GRH for the Dedelkind zeta function $\zeta_K(s)$. Then for any ideal $\mathfrak{c}$ coprime with $\mathfrak{m}$, $$\theta(x;\mathfrak{c},\mathfrak{m})=\frac{1}{h_{\mathfrak{m}}}x+O_{K,\mathfrak{m}}(\sqrt{x}(\log x)^2),$$as $x\rightarrow+\infty$.
\end{thm}
\begin{thm}\label{thm4}
	If $d>1$, then $\zeta_K(s)$ has at most one zero $\beta_0=\sigma+it$ in the region$$\sigma\ge 1-(4\log D_K)^{-1},\ |t|\le(4\log D_K)^{-1}.$$If $\beta_0$ exists, it is real and simple. There exist effectively computable absolute constants such that$$\theta(x;\mathfrak{c},\mathfrak{m})=\frac{1}{h_{\mathfrak{m}}}x+O_{K,\mathfrak{m}}(xe^{-C\sqrt{\log x}}+x^{\beta_0})$$as $x\rightarrow+\infty$, with the understanding that the $\beta_0$ term is present only if $\beta_0$ exists. 
\end{thm}
If $\mathfrak{a}\in I^{S(\mathfrak{m})}$. Let $S(M)$ be the set of ideals of $K$ defined by$$S(M)=S_{\mathfrak{a},\mathfrak{m}}(M)=\{I\subseteq\mathcal{O}_K:N_{K/\Q}I\le M,\ I\sim\mathfrak{a}\mbox{ mod }C_{\mathfrak{m}}\}.$$
\begin{prop}\label{prop2}
		\begin{enumerate}
		\item[$\paruno \ai\pardue$] Assume the GRH for $\zeta_K(s)$. Then$$\log|N_{K/\Q}(\lcm(S(M)))|=M\cdot\frac{1}{h_{\mathfrak{m}}}\sum_{\mathfrak{c}\in C_{\mathfrak{m}}}\frac{1}{N_{K/\Q}\mathfrak{c}}+O_{K,\mathfrak{m}}(\sqrt{M}(\log M)^2),$$as $M\rightarrow+\infty$, where $N_{K/\Q}\mathfrak{c}$ is the smallest norm of an integral ideal in the ray class group of $\mathfrak{c}$.
	\item[$\paruno \bi\pardue$] Let $\beta_0$ be the possible Siegel zero of $\zeta_K(s).$ Then$$\log|N_{K/\Q}(\lcm(S(M)))|=M\cdot\frac{1}{h_{\mathfrak{m}}}\sum_{\mathfrak{c}\in C_{\mathfrak{m}}}\frac{1}{N_{K/\Q}\mathfrak{c}}+O_{K,\mathfrak{m}}(\sqrt{M}e^{-C\sqrt{\log M}}+M^{\beta_0}),$$as $M\rightarrow+\infty$, where the last error term is present if and only if $\beta_0$ exists.
	\end{enumerate}
\end{prop}
\begin{proof}
	Let $P(M)$ be the set of prime factors of $\lcm(S(M))$, that is$$P(M)=\{\wp\in\mathcal{O}_K:\wp\mbox{ divides at least one }I\subseteq S(M)\}.$$We are going to characterize the primes in the set $P(M)$. \\
	
	We have that$$N_{K/\Q}(\lcm(S(M)))=\prod_{\wp\in P(M)}q_\wp^{\beta_{\wp}(M)},$$where $\beta_{\wp}(M)$ is the highest power of $\wp$ dividing any element of $S(M)$. On the other hand, if we let $T(M)=\prod_{\wp\in P(M)}q_\wp$, it turns out that$$\frac{N_{K/\Q}(\lcm(S(M)))}{N_{K/\Q}(T(M))}=\prod_{\wp}q_\wp^{\beta_{\wp}(M)-1},$$where the product is over the primes $\wp$ whose square divides some element of $S(M)$. In particular, each of those $\wp$ has norm $q_\wp\le M$ and they are at most $\sqrt{M}$. Therefore$$\log|N_{K/\Q}(\lcm(S(M)))|-\log|N_{K/\Q}(T(M))|\ll\sqrt{M}\log M. $$Then it suffices to focus on $\log|N_{K/\Q}(T(M))|$.\\
	Note that if $I\in S(M)$, then $I\in S(\mathfrak{m})$, so it is coprime with $\mathfrak{m}$. Let now $\wp\supseteq I$ (i.e. $\wp\in P(M))$; then $\wp\equiv\mathfrak{c}$ in $C_{\mathfrak{m}}$ for some representative $\mathfrak{c}\in I^{S(\mathfrak{m})}$. In particular $\mathfrak{c}^{-1}\in I^{S(\mathfrak{m})}$; let then $\mathfrak{d}=\mathfrak{c}^{-1}\mathfrak{a}$, so that $\mathfrak{c}\mathfrak{d}\equiv\mathfrak{a}$ in $C_{\mathfrak{m}}$. We can take $\mathfrak{d}$ of smallest norm among the coset representatives. Since every class in $C_{\mathfrak{m}}$ is represented by an integral ideal, we can also assume $\mathfrak{d}$ integral ideal of smallest norm.\\
	Now, $\mathfrak{d}\wp\in S(M)$ is equivalent to $N_{K/\Q}(\mathfrak{d}\wp)=N_{K/\Q}\mathfrak{d}\cdot q_{\wp}\le M$. Hence $\wp\equiv\mathfrak{c}$ in $C_{\mathfrak{m}}$ is in $P(M)$ if and only if$$q_{\wp}\le\frac{M}{N_{K/\Q}\mathfrak{d}}.$$Denote by $U(M)$ the set of such primes.\\
	Finally,\begin{align*}		
		\log|N_{K/\Q}(T(M))|&=\sum_{\wp\in P(M)}\log q_\wp\\
		&=\sum_{\mathfrak{c}\in C_{\mathfrak{m}}}\sum_{\wp\in U(M)}\log q_{\wp}\\
		&=\sum_{\mathfrak{c}\in C_{\mathfrak{m}}}\theta\left( \tfrac{M}{N_{K/\Q}\mathfrak{d}};\mathfrak{c};\mathfrak{m}\right) \\
		&=\frac{M}{h_{\mathfrak{m}}}\sum_{\mathfrak{c}\in C_{\mathfrak{m}}}\frac{1}{N_{K/\Q}\mathfrak{c}}+o(M),
	\end{align*}where $N_{K/\Q}\mathfrak{c}$ is by definition the smallest norm of an integral ideal in the ray class group of $\mathfrak{c}$. The last equality holds because $\mathfrak{d}$ runs over $C_{\mathfrak{m}}$ as $\mathfrak{c}$ does.
	The result follows by applying either Theorem \ref{thm3} or Theorem \ref{thm4} to the last step.
	
\end{proof}

Let now the modulus $\mathfrak{m}$ be a "principal modulus", that is $\mathfrak{m}=(\nu)$, where $\nu\in\mathcal{O}_K$, and let $\alpha\in\mathcal{O}_K$ be coprime with $\nu$.
Consider the set of principal ideals$$S(M)=\{(\beta)\in\mathcal{O}_K:N_{K/\Q}\beta\le M,\ (\beta)\sim(\alpha)\mbox{ mod }C_{(\nu)}\}.$$
It is well-known that $(\beta)\sim(\alpha)$ mod $C_{(\nu)}$ is equivalent to say that there exist $a,b\in\mathcal{O}_K$, $a,b\neq0$, $a\equiv b\equiv1$ mod $\nu$ so that $(\beta)b=(\alpha)a$. It follows that$$(\beta)\sim(\alpha)\mbox{ mod }C_{(\nu)}\mbox{ if and only if }\beta\equiv\eta\alpha\mbox{ mod }\nu\mbox{ for some }\eta\in\mathcal{O}_K^{\times}.$$In particular
$$S(M)\sim\{f_{\eta}(\lambda):\eta\in\mathcal{O}_K^\times,\ \lambda\in\mathcal{O}_K,\ N_{K/\Q}\lambda\ll M\},$$where $f_{\eta}(X)=\eta\alpha+\nu X\in\mathcal{O}_K[X]$. It is immediate from Proposition \ref{prop2} the following corollary, which can be interpreted as a version of Cilleruelo's conjecture for linear polynomials over $K$.
\begin{coroll}
	\begin{multline*}
		\log|N_{K/\Q}(\lcm(f_{\eta}(\lambda):\eta\in\mathcal{O}_K^\times,\ \lambda\in\mathcal{O}_K,\ N_{K/\Q}\lambda\ll M))|\\=M\cdot\frac{1}{h_{(\nu)}}\sum_{\mathfrak{c}\in C_{\mathfrak{m}}}\frac{1}{N_{K/\Q}\mathfrak{c}}+o(M),
	\end{multline*}as $M\rightarrow+\infty$.
\end{coroll}

\begin{flushleft}
	\textbf{Example:} In the ring of 
	Eisenstein integers $\Z[\omega]=\mathcal{O}_{\Q(\omega)}$, $w=e^{2\pi i/3}$, all ideals are principal. In the above notations we have $(\alpha)\sim(\beta)$ mod $\nu$ iff $\alpha=\eta\beta$, where $\eta=\pm\omega^j\in\Z[\omega]^\times$, $j=0,1,2$.
	In this case we get
\end{flushleft}\begin{multline*}
\log|N_{K/\Q}(\lcm(\eta\alpha+\nu\lambda:\eta\in\{\pm1,\pm\omega,\pm\omega^2\},\  \lambda\in\mathcal{O}_K,\ N_{K/\Q}\lambda\le M))|\\=M\cdot\frac{1}{h_{(\nu)}}\sum_{\mathfrak{c}\in C_{\mathfrak{m}}}\frac{1}{N_{K/\Q}\mathfrak{c}}+o(M),
\end{multline*}as $M\rightarrow+\infty$.

\subsection{Average version for irreducible $S_n$-polynomials of higher degree}
\subsubsection*{Proof of Theorem \ref{thm2}}
As already mentioned, almost all $f\in\mathscr{P}_{n,N}^{0}(K)$ are irreducible (the error term in this case is $O(N^{-d}))$.
Assume now $n\ge3$. We start by computing an asymptotic formula for the mean value of the quantity$$\Psi_f(N,M)=\log|N_{K/\Q}(\lcm(f(\lambda):\lambda\in\mathcal{O}_K,\ N_{K/\Q}\lambda\le M))|.$$

\begin{prop}\label{prop3}
	Let $N,M>0$ such that$$M(\log M)^\ell\ll N=o\left( M\frac{\log M}{\log\log M}\right) $$for some $0<\ell<1$. Then\begin{multline*}
		\mathbb{E}_N(\Psi_f(N,M))=(n-1)M\log M
		+O\left( M\frac{\log M}{\log\log M}+N\log\log M\right),
	\end{multline*}as $N,M\rightarrow+\infty$.
\end{prop}

\begin{proof}
	Following \cite{Cil}, we compare the behaviour of$$\lcm(f(\lambda):N_{K/\Q}\lambda\le M)=\prod_{\wp\in\mathcal{P}_f}\wp^{\beta_\wp(M)}$$and$$P_f(M):=\prod_{N_{K/\Q}\lambda\le M}|N_{K/\Q}f(\lambda)|=\prod_{\wp}|N_{K/\Q}\wp|^{\alpha_\wp(M)},$$where $ \mathcal{P}_f $ is the set of primes such that the equation $ f\equiv 0\mod \wp $ has some solutions, which is the set of $ \wp $ so that $ \Frob_{f,\wp}\in G_f $ has fixed points. We start by writing\begin{align*}
		\Psi_f(N,M)&=\log P_f(M)+\sum_{N_{K/\Q}\wp\le M}\beta_\wp(M)\log N_{K/\Q}\wp \\
		&-\underset{\wp\tiny\mbox{ unramified}}{\sum_{N_{K/\Q}\wp\le M}}\alpha_\wp(M)\log N_{K/\Q}\wp\\&-\underset{\wp\tiny\mbox{ ramified}}{\sum_{N_{K/\Q}\wp\le M}}\alpha_\wp(M)\log N_{K/\Q}\wp\\&-\sum_{N_{K/\Q}\wp>M}(\alpha_\wp(M)-\beta_\wp(M))\log N_{K/\Q}\wp
	\end{align*}and we're going to study all these five terms.
	
	$\bullet$ $\log P_f(M)=\sum_{N_{K/\Q}\lambda\le M}\log|N_{K/\Q}f(\lambda)|$; pick $A=A(M,N)$ such that $A=o(M) $ and $ A\gg\frac{N}{\log M} $. 
	
	Then for $A \ll N_{K/\Q}\lambda\le M $ and $ f(X)=X^{n}+\sum_{i=0}^{n-1}\alpha_{n-i-1}X^{n-i-1} $, one has\begin{align*}
		\log|N_{K/\Q}f(\lambda)|&=n\log |N_{K/\Q}\lambda|+\log\left|N_{K/\Q}\left( 1+\frac{\alpha_{n-1}}{\lambda}+\dots+\frac{\alpha_0}{\lambda^n}\right)\right| \\
		&=n\log |N_{K/\Q}\lambda|+\log\prod_{i=1}^{d}\left| 1+\sigma_i\left( \frac{\alpha_{n-1}}{\lambda}+\dots\right) \right| \\
		&=n\log |N_{K/\Q}\lambda|+\sum_{i=1}^{d}\log\left( \left| 1+\sigma_i\left( \frac{\alpha_{n-1}}{\lambda}+\dots\right) \right|\right)  \\
		&= n\log |N_{K/\Q}\lambda|+\sum_{i=1}^{d}O\left( \sigma_i\left( \frac{\alpha_{n-1}}{\lambda}\right) +\dots+\sigma_i\left( \frac{\alpha_0}{\lambda^n}\right) \right)\\
		&= n\log |N_{K/\Q}\lambda|+\sum_{i=1}^{d}O\left( \frac{N}{N_{K/\Q}\lambda}+\dots+\frac{N}{N_{K/\Q}\lambda^n}\right) \\
		&= n\log |N_{K/\Q}\lambda|+O_{n,K}\left( \frac{N}{A}\right) ,
	\end{align*}where $\sigma_1,\dots,\sigma_d$ are the $\Q$-embeddings of $K$ into $\Co$.
	
	If $ 1\le N_{K/\Q}\lambda\ll A $, we simply use that $ |N_{K/\Q}f(\lambda)|\ll N^d M^n $, so $ \log|N_{K/\Q}f(\lambda)|\ll_{n,d} \log N +\log M$. Therefore, since the elements in $\mathcal{O}_K$ of norm at most $M$ are at most $M$,\begin{multline*}
		\log P_f(M)=\sum_{ A\ll N_{K/\Q}\lambda\le M}\log|N_{K/\Q}f(\lambda)|+\sum_{N_{K/\Q}\lambda\ll A}\log|N_{K/\Q}f(\lambda)|\\
		=\sum_{ A\ll N_{K/\Q}\lambda\le M}\left( n\log N_{K/\Q}\lambda+O\left( \frac{N }{A}\right) \right) +\sum_{N_{K/\Q}\lambda\ll  A}\log |N_{K/\Q}f(\lambda)|\\
		=n M\log M+O\left( M+\frac{NM}{A}+A(\log N+\log M)\right)\\
		=n M\log M+O\left( M\frac{\log M}{\log\log M}+N\log\log M\right),
	\end{multline*}as $M\rightarrow+\infty$, by choosing $A=\frac{N}{\log M}\log\log M$ and $N=o\left( M\frac{\log M}{\log\log M}\right)$.
	
	$ \bullet $ $ \beta_\wp(N)=\underset{N_{K/\Q}\lambda\le M}{\max}\max\{k\ge0:\wp^k|f(\lambda)\} $; if $ \wp^k|f(\lambda) $, then in particular $ k\le\frac{\log|N_{K/\Q}f(\lambda)|}{\log q_\wp}\ll\frac{\log N+\log M}{\log q_{\wp}} $. Thus\begin{align*}
		\sum_{q_{\wp}\le M}\beta_\wp(M)\log q_\wp&\ll \sum_{q_\wp\le M}(\log N+\log M)\\
		&\ll M\Big(1+\frac{\log N}{\log M}\Big)\ll M
	\end{align*}under the conditions above.
	
	$\bullet $ If $ \wp $ is a prime which doesn't divide $ \mathfrak{D}_{K_f/K} $, then the number of solutions $ s_{\wp^k}(f) $ of $ f \mod \wp^k $ is equal to the number $ s_\wp(f) $ of solutions mod $ \wp $ (see Theorem 1 of \cite{Na}). On the other hand, by dividing the interval $ [1,M] $ into consecutive intervals of length $ q_{\wp}^k $, one has$$s_{\wp^k}(f)\Big[\frac{M}{q_{\wp}^k}\Big]\le \underset{f(\lambda)\equiv 0\ (\wp^k)}{\sum_{N_{K/\Q}\lambda\le M}}1\le s_{\wp^k}(f)\Big(\Big[\frac{M}{q_{\wp}^k}\Big]+1\Big),$$so$$\underset{f(\lambda)\equiv 0\ (\wp^k)}{\sum_{N_{K/\Q}\lambda\le M}}1=M\frac{s_{\wp^k}(f)}{q_{\wp}^k}+O(s_{\wp^k}(f)).$$ For those $ \wp $, one has\begin{align*}
		\alpha_\wp(M)&=\sum_{N_{K/\Q}\lambda\le M}\underset{\wp^k| f(\lambda)}{\sum_{k\ge 1}}1=\sum_{k\ge 1}\underset{f(\lambda)\equiv 0\ (\wp^k)}{\sum_{N_{K/\Q}\lambda\le M}}1\\
		&=\sum_{k\ge 1}M\frac{s_{\wp^k}(f)}{q_\wp^k}+O\Big(\sum_{1\le k\le \frac{\log N+\log M}{\log q_\wp}}1\Big)\\
		&=M\frac{s_\wp(f)}{q_\wp-1}+O\Big(\frac{\log N}{\log q_\wp}+\frac{\log M}{\log q_\wp}\Big).
	\end{align*}Therefore$$\underset{\wp\tiny\mbox{ unramified}}{\sum_{q_\wp\le M}}\alpha_\wp(M)\log q_\wp=M\underset{\wp\tiny\mbox{ unramified}}{\sum_{q_\wp\le M}}\frac{\log q_\wp}{q_\wp-1}s_\wp(f)+O(M).$$Using Proposition \ref{p9} we can estimate on average $ \underset{\wp\tiny\mbox{ unramified}}{\sum_{q_\wp\le x}}s_\wp(f) $ for $ x>0 $, $ x<N^{d\xi/(n+1)} $:\begin{align*}
		\frac{1}{|\mathscr{P}_{n,N}^{0}|}\sum_{f\in\mathscr{P}_{n,N}^{0}} \underset{\wp\tiny\mbox{ unramified}}{\sum_{q_\wp\le x}}s_\wp(f)&=\frac{1}{|\mathscr{P}_{n,N}^{0}|}\sum_{f\in\mathscr{P}_{n,N}^{0}} \underset{\wp\tiny\mbox{ unram.}}{\sum_{q_\wp\le x}}\underset{f(\alpha)\equiv 0\ (\wp)}{\sum_{\alpha\tiny{\mbox{ mod }}\wp}}1\\
		&=\sum_{\alpha}\underset{\sigma \alpha=\alpha}{\sum_{ \sigma\in S_n}}\frac{1}{|\mathscr{P}_{n,N}^{0}|}\sum_{f\in\mathscr{P}_{n,N}^{0}}\underset{\Frob_{f,\wp}=\sigma}{\sum_{q_\wp\le x,\ \wp\tiny\mbox{ unram.}}}1\\
		&=\sum_{\alpha}\underset{\sigma \alpha=\alpha}{\sum_{ \sigma\in S_n}}\frac{1}{|\mathscr{P}_{n,N}^{0}|}\sum_{f\in\mathscr{P}_{n,N}^{0}}\pi_{\mathscr{C}(\sigma),K_f/K}(x)\\
		&=\sum_{\alpha}\underset{\sigma \alpha=\alpha}{\sum_{ \sigma\in S_n}}\mathbb{E}_N(\pi_{\mathscr{C}(\sigma),K_f/K}(x))\\
		&=\sum_{\alpha}\underset{\sigma \alpha=\alpha}{\sum_{ \sigma\in S_n}}\Big(\frac{|\mathscr{C}(\sigma)|}{n!}\pi_K(x)+O(\log\log x)\Big)\\
		&=\pi_K(x)+O(\log\log x),
	\end{align*}where $\pi_{\mathscr{C}(\sigma),K_f/K}$ is the Chebotarev Density Theorem function on the conjugacy class $\mathscr{C}(\sigma)$ of $\sigma$. Note that$$\pi_{\mathscr{C}(\sigma),K_f/K}-\pi_{f,r}(x)\ll_{n,K}\log\log x$$on average, if $\mathscr{C}(\sigma)=\mathscr{C}_r$ for some $r$. Write$$s_\wp(f)=1+\sigma_\wp(f),$$where $ -1\le\sigma_\wp(f)\le n-1 $ and $$\frac{1}{|\mathscr{P}_{n,N}^{0}|}\sum_{f\in\mathscr{P}_{n,N}^{0}} \underset{\wp\tiny\mbox{ unramified}}{\sum_{q_\wp\le x}}\sigma_\wp(f)\ll_{n,K}\log\log x,$$ if $ x<N^{d\xi/(n+1)}  $.
	
	Now,\begin{align*}
		\underset{\wp\tiny\mbox{ unramified}}{\sum_{q_\wp\le M}}\frac{\log q_\wp}{q_\wp-1}s_\wp(f)&=\sum_{ q_\wp\le M}\frac{\log q_\wp}{q_\wp}-\underset{\wp\tiny\mbox{ ramified}}{\sum_{q_\wp\le M}}\frac{\log q_\wp}{q_\wp}\\&+\underset{\wp\tiny\mbox{ unramified}}{\sum_{q_\wp\le M}}\frac{\log q_\wp}{q_\wp}\sigma_\wp(f)+O(1).
	\end{align*}Since$$\sum_{ q_\wp\le M}\frac{\log q_\wp}{q_\wp}=\log M+O(1),$$and$$\underset{\wp\tiny\mbox{ ramified}}{\sum_{q_\wp\le M}}\frac{\log q_\wp}{q_\wp}\ll \log\log|N_{K/\Q}\mathfrak{D}_{K_f/K}|\ll\log\log N$$(see \cite{RZ}, Lemma 3.2), one gets$$\underset{\wp\tiny\mbox{ unramified}}{\sum_{q_\wp\le M}}\frac{\log q_\wp}{q_\wp-1}s_\wp(f)=\log M+\underset{\wp\tiny\mbox{ unramified}}{\sum_{q_\wp\le M}}\frac{\log q_\wp}{q_\wp}\sigma_\wp(f)+O(\log\log N).$$
	Let $ 0<\delta<\frac{1}{2(n+1)} $ and $ N>M(\log M)^{2\delta(n+1)} $, so that$$M':=\frac{M^{1/2(n+1)}(\log M)^{\delta}}{(\log N)^{1/(n+1)}}<N^{d\xi/(n+1)}.$$Write$$\underset{\wp\tiny\mbox{ unramified}}{\sum_{q_\wp\le M}}\frac{\log q_\wp}{q_\wp}\sigma_\wp(f)=\underset{\wp\tiny\mbox{ unramified}}{\sum_{q_\wp\le M'}}\frac{\log q_\wp}{q_\wp}\sigma_\wp(f)+\underset{\wp\tiny\mbox{ unramified}}{\sum_{M'<q_\wp\le M}}\frac{\log q_\wp}{q_\wp}\sigma_\wp(f).$$
	For the first term, by partial integration we obtain\begin{multline*}
		\frac{1}{|\mathscr{P}_{n,N}^{0}|}\sum_{f\in\mathscr{P}_{n,N}^{0}}\underset{\wp\tiny\mbox{ unramified}}{\sum_{q_\wp\le M'}}\frac{\log q_\wp}{q_\wp}\sigma_\wp(f)\ll\frac{\log M'}{M'}\log\log M'\\+\int_{2}^{M'}\log\log t\frac{(1-\log t)}{t^2}dt\ll 1,
	\end{multline*}since $ \int_{2}^{M'}\log\log t\frac{(1-\log t)}{t^2}dt\ll\int_{2}^{M'}\frac{t^{1/2}}{t^2}dt\ll 1 $.
	
	To treat the second term, note that it is$$\le(n-1)\sum_{M'<q_\wp\le M}\frac{\log q_\wp}{q_\wp}\le(n-1)\sum_{M-y<q_\wp\le M}\frac{\log q_\wp}{q_\wp},$$for $ y\ge M-M' $. If moreover we pick $ M\sim 2y $, then$$\sum_{M-y<q_\wp\le M}\frac{\log q_\wp}{q_\wp}\ll\log M-\log (M-y)\ll1,$$as $ M\rightarrow+\infty $.

	Hence we have the following estimate on average:\begin{align*}
		\frac{1}{|\mathscr{P}_{n,N}^{0}|}\sum_{f\in\mathscr{P}_{n,N}^{0}}\underset{\wp\tiny\mbox{ unramified}}{\sum_{q_\wp\le M}}\alpha_\wp(M)\log q_\wp&=M\log M+O\Big(M\log\log N+M\Big)\\&=M\log M+O(M\log\log N),
	\end{align*}for $  N>M(\log M)^{2\delta(n+1)} $, $ 0<\delta<\frac{1}{2(n+1)} $.
	
	$\bullet$ We divide the sum into two terms:\begin{align*}
		\underset{\wp\tiny\mbox{ ramified}}{\sum_{q_\wp\le M}}\alpha_\wp(M)\log q_\wp&=\underset{\wp\tiny\mbox{ ramified}}{\sum_{q_\wp\le M}}\log q_\wp|\{\lambda\in\mathcal{O}_K:N_{K/\Q}\lambda\le M,\ f(\lambda)\equiv 0\mbox{ mod } \wp\}|\\
		&+\underset{\wp\tiny\mbox{ ramified}}{\sum_{q_\wp\le M}}\log q_\wp\sum_{ N_{K/\Q}\lambda\le M}\underset{f(\lambda)\equiv 0\ (\wp^k)}{\sum_{k\ge 2}}1\\
		&=\mbox{I}+\mbox{II}.
	\end{align*}
	To estimate I, note that$$|\{\lambda\in\mathcal{O}_K:N_{K/\Q}\lambda\le M,\ f(\lambda)\equiv 0\mbox{ mod } \wp\}|=\Big[\frac{M}{q_\wp}\Big]s_\wp(f)\ll\frac{M}{q_\wp}s_\wp(f),$$so\begin{align*}
		\mbox{I}\ll\underset{\wp\tiny\mbox{ ramified}}{\sum_{q_\wp\le M}}\frac{\log q_\wp}{q_\wp}s_\wp(f)
		\ll_n M\underset{\wp\tiny\mbox{ ramified}}{\sum_{q_\wp\le M}}\frac{\log q_\wp}{q_\wp}
		\ll M\log\log N.
	\end{align*}The mean value of II is$$
	\frac{1}{|\mathscr{P}_{n,N}^{0}|}\sum_{f\in\mathscr{P}_{n,N}^{0}} \mbox{II}\ll\frac{1}{N^{nd}}\sum_{ q_\wp\le M}\log q_\wp\sum_{ N_{K/\Q}\lambda\le M}\sum_{2\le k\ll \frac{\log N+\log M}{\log q_\wp}}\underset{f(\lambda)\equiv 0\ (\wp^k)}{\sum_{f\in\mathscr{P}_{n,N}^{0}}}1.$$Similarly as we computed in Lemma \ref{l11}, note that for any $\lambda\in\mathcal{O}_K$,\begin{align*}
		\underset{f(\lambda)\equiv 0\ (\wp^k)}{\sum_{f\in\mathscr{P}_{n,N}^{0}}}1&=\underset{g(\lambda)=0}{\sum_{g\in\mathbb{F}_{q_{\wp}^k}[X]}}\underset{f\equiv g\tiny\mbox{ mod }\wp^k}{\sum_{f\in \mathscr{P}_{n,N}^{0}}}1.
	\end{align*}Since there are $q_{\wp}^{k(n-2)}$ possibilities for $g$ as in the above sum, one has$$\underset{f(\lambda)\equiv 0\ (\wp^k)}{\sum_{f\in\mathscr{P}_{n,N}^{0}}}1=\frac{(2N)^{nd}}{q_\wp^{2k}}+O(N^{d(n-\xi)})$$as long as $k\ll\frac{\log N}{\log q_\wp}$. Hence\begin{multline*}
		\frac{1}{|\mathscr{P}_{n,N}^{0}|}\sum_{f\in\mathscr{P}_{n,N}^{0}} \mbox{II}\ll M\sum_{ q_\wp\le M}\log q_\wp\sum_{k\ge 2}\Big(\frac{1}{q_{\wp}^2}\Big)^k+\frac{M}{N^{d\xi}}\sum_{ q_\wp\le M}\log q_\wp\sum_{k\ll \frac{\log N+\log M}{\log q_\wp}}1\\
		\ll M+\frac{M^2}{N^{d\xi}}\Big(\frac{\log N}{\log M}+1\Big).
	\end{multline*}To conclude$$\frac{1}{|\mathscr{P}_{n,N}^{0}|}\sum_{f\in\mathscr{P}_{n,N}^{0}}\underset{\wp\tiny\mbox{ ramified}}{\sum_{q_\wp\le M}}\alpha_\wp(M)\log q_\wp\ll M\log\log N+\frac{M^2}{N^{d\xi}}\Big(\frac{\log N}{\log M}+1\Big).$$
	
	$\bullet$ For $\lambda,\mu\in\mathcal{O}_K$ such that $N_{K/\Q}\lambda<N_{K/\Q}\mu$ let$$G(\mu,\lambda)=\frac{f(\mu)-f(\lambda)}{\mu-\lambda}.$$Once fixed $ \mu $, $ G(\mu,\lambda) $ is a polynomial in $ \lambda $ of degree $ n-1 $. 
	
	We are now dealing with the primes $\wp$ of norm $ q_\wp>M $, for which\begin{align*}
		\alpha_\wp(M)&=\sum_{ N_{K/\Q}\lambda\le M}\sum_{k\ge 1}\mathbbm{1}(f(\lambda)\equiv 0\mbox{ mod }\wp^k)\\
		&=\sum_{ k\ge 1}\underset{f(\lambda)\equiv 0\ (\wp^k)}{\sum_{ N_{K/\Q}\lambda\le M}}1\ll\sum_{1\le k\ll \frac{\log N+\log M}{\log q_\wp}}\ll_{n,K} 1.
	\end{align*}For $\wp$ of norm $ q_\wp>M $ we then have$$\alpha_\wp(M)-\beta_\wp(M)\ll_{n,K} 1.$$Note also that if $ \wp|f(\lambda) $, then $ |q_\wp|\le|N_{K/\Q}f(\lambda)|\ll N^dM^n $, so $ \alpha_\wp(M)=0 $ for $ q_\wp\gg N^dM^n $. Also, $ \alpha_\wp(M)\neq\beta_\wp(M) $ if and only if there exist $\mu,\lambda\in\mathcal{O}_K $, $N_{K/\Q}\lambda< N_{K/\Q}\mu\le M $ such that $ \wp|f(\mu) $ and $ \wp|f(\lambda) $, equivalently $ \wp|f(\lambda) $ and $ \wp|(\mu-\lambda)G(\mu,\lambda) $; but $ \wp\nmid(\mu-\lambda) $, since $ |N_{K/\Q}(\mu-\lambda)|\le M-1<q_\wp $, so $ \wp|G(\mu,\lambda) $.
	
	Thefore\begin{multline*}
		\sum_{ q_\wp> M}(\alpha_\wp(M)-\beta_\wp(M))\log q_\wp\ll\sum_{1\le N_{K/\Q}\lambda<N_{K/\Q}\mu\le M}\underset{\wp|G(\mu,\lambda)}{\underset{\wp|f(\ell)}{\sum_{ M<q_\wp\ll N^dM^n}}}\log q_\wp\\
		=\underset{G(\mu,\lambda)=0}{\sum_{1\le N_{K/\Q}\lambda<N_{K/\Q}\mu\le M}}\underset{\wp|f(\lambda)}{\sum_{ M<q_\wp\ll N^dM^n}}\log q_\wp+\underset{G(\mu,\lambda)\neq 0}{\sum_{1\le N_{K/\Q}\lambda<N_{K/\Q}\mu\le M}}\underset{\wp|G(\mu,\lambda)}{\underset{\wp|f(\lambda)}{\sum_{ M<q_\wp\ll N^dM^n}}}\log q_\wp\\
		\ll\sum_{1\le N_{K/\Q}\lambda<N_{K/\Q}\mu\le M}\underset{\wp|f(\lambda)}{\sum_{ M<q_\wp\ll N^dM^n}}\log q_\wp+\underset{G(\mu,\lambda)\neq 0}{\sum_{1\le N_{K/\Q}\lambda<N_{K/\Q}\mu\le M}}\underset{\wp|G(\mu,\lambda)}{\underset{\wp|f(\lambda)}{\sum_{ M<q_\wp\ll N^dM^n}}}\log q_\wp\\
		\ll(\log N+\log M)\underset{N_{K/\Q}\mu\le M}{\max}\{\wp:q_\wp>M,\ \wp|f(\mu)\}\\+\underset{G(\mu,\lambda)\neq 0}{\sum_{1\le N_{K/\Q}\lambda<N_{K/\Q}\mu\le M}}\underset{\wp|G(\mu,\lambda)}{\underset{\wp|f(\lambda)}{\sum_{ M<q_\wp\ll N^dM^n}}}\log q_\wp.
	\end{multline*}For $ N_{K/\Q}\mu\le M$, $ |N_{K/\Q}f(\mu)|\ll N^d M^n $, so the primes $\wp$ with $ q_\wp>M $ dividing $ f(\mu) $ are at most $ \ll\frac{\log (N^dM^n)}{\log M}\ll_{n,K}1 $. Thus$$\sum_{ q_\wp> M}(\alpha_\wp(M)-\beta_\wp(M))\log q_\wp\ll \underset{G(\mu,\lambda)\neq 0}{\sum_{1\le N_{K/\Q}\lambda<N_{K/\Q}\mu\le M}}\underset{\wp|G(\mu,\lambda)}{\underset{\wp|f(\lambda)}{\sum_{ M<q_\wp\ll N^dM^n}}}\log q_\wp+\log M,$$or on average\begin{multline*}
		\frac{1}{|\mathscr{P}_{n,N}^{0}|}\sum_{ q_\wp> M}(\alpha_\wp(M)-\beta_\wp(M))\log q_\wp\\
		\ll\underset{G(\mu,\lambda)\neq 0}{\sum_{1\le N_{K/\Q}\lambda<N_{K/\Q}\mu\le M}}\underset{\wp|G(\mu,\lambda)}{\sum_{ M<q_\wp\ll N^dM^n}}\log q_\wp|\{f:f(\lambda)\equiv 0\mbox{ mod }\wp\}|+\log M\\
		=\underset{G(\mu,\lambda)\neq 0}{\sum_{1\le N_{K/\Q}\lambda<N_{K/\Q}\mu\le M}}\underset{\wp|G(\mu,\lambda)}{\sum_{ M<q_\wp\ll N^dM^n}}\log q_\wp\Big(\frac{1}{q_\wp^2}+O\Big(\frac{1}{N^{d\xi}}\Big)\Big)+\log M\\
		\ll \underset{G(\mu,\lambda)\neq 0}{\sum_{1\le N_{K/\Q}\lambda<N_{K/\Q}\mu\le M}}\underset{\wp|G(\mu,\lambda)}{\sum_{ M<q_\wp\ll N^dM^n}}\frac{\log q_\wp}{q_\wp^2}\\+\frac{1}{N^{d\xi}}\underset{G(\mu,\lambda)\neq 0}{\sum_{1\le N_{K/\Q}\lambda<N_{K/\Q}\mu\le M}}\underset{\wp|G(\mu,\lambda)}{\sum_{ M<q_\wp\ll N^dM^n}}\log q_\wp+\log M\\
		=\mbox{I}+\mbox{II}+\log M.
	\end{multline*}For II, observe that since $ |G(\mu,\lambda)|\ll N^dM^{n-1} $, the number of primes $\wp$ of norm $ q_\wp>M $ dividing $ G(\mu,\lambda) $ is at most $ \ll\frac{\log(N^dM^{n-1})}{\log M}\ll1 $, so$$\mbox{II}\ll\frac{M^2}{N^{d\xi}}\log M.$$For I, we separate the contribution of small and large prime. Pick $ M<B_{M,N}\ll N^dM^n $; for small primes we have\begin{align*}
		\underset{G(\mu,\lambda)\neq 0}{\sum_{1\le N_{K/\Q}\lambda<N_{K/\Q}\mu\le M}}\underset{\wp|G(\mu,\lambda)}{\sum_{ M<q_\wp\le B_{M,N}}}\frac{\log p}{q_\wp^2}&=\sum_{ M<q_\wp\le B_{M,N}}\frac{\log q_\wp}{q_\wp^2}\underset{G(\mu,\lambda)\equiv 0\ (\wp)}{\sum_{1\le N_{K/\Q}\lambda<N_{K/\Q}\mu\le M}}1\\
		&\ll M \sum_{ M<q_\wp\le B_{M,N}}\frac{\log q_\wp}{q_\wp^2}\ll M,
	\end{align*}since $ \underset{G(\mu,\lambda)\equiv 0\ (\wp)}{\sum_{1\le N_{K/\Q}\lambda<N_{K/\Q}\mu\le M}}1\le (n-1)M $. For large primes,\begin{multline*}
		\underset{G(\mu,\lambda)\neq 0}{\sum_{1\le N_{K/\Q}\lambda<N_{K/\Q}\mu\le M}}\underset{\wp|G(\mu,\lambda)}{\sum_{ B_{M,N}<q_\wp\ll N^dM^n }}\frac{\log q_\wp}{q_\wp^2}\\
		\ll\frac{(\log N+\log M)}{B_{M,N}^2}\underset{G(\mu,\lambda)\neq 0}{\sum_{1\le N_{K/\Q}\lambda<N_{K/\Q}\mu\le M}}|\{\wp:q_\wp>B_{M,N},\ \wp|G(\mu,\lambda)\}|\\
		\ll\frac{M^2}{B_{M,N}^2}\log M\frac{\log M}{\log B_{M,N}},
	\end{multline*}by observing that $ |\{\wp:q_\wp>B_{M,N},\ \wp|G(\mu,\lambda)\}|\ll \frac{\log(N^dM^{n-1})}{\log B_{M,N}}\ll\frac{\log M}{\log B_{M,N}} $ since $ |G(\mu,\lambda)|\ll N^dM^{n-1} $.
	We obtained$$
	\frac{1}{|\mathscr{P}_{n,N}^{0}|}\sum_{f\in\mathscr{P}_{n,N}^{0}}\sum_{ q_\wp> M}(\alpha_\wp(M)-\beta_\wp(M))\log q_\wp\ll M+\frac{M^2}{N^{d\xi}}\log M+\log M
	$$by choosing for instance $ B_{M,N}=M\log M $.

	Finally,\begin{multline*}
		\frac{1}{|\mathscr{P}_{n,N}^{0}|}\sum_{f\in\mathscr{P}_{n,N}^{0}}\log|N_{K/\Q}(\lcm(f(\lambda):N_{K/\Q}\lambda\le M))=(n-1)M\log M\\
		+O\left( M\frac{\log M}{\log\log M}+N\log\log M +M\log\log M+\frac{M^2}{N^{d\xi}}\log M\right)\\
		=(n-1)M\log M
		+O\left( M\frac{\log M}{\log\log M}+N\log\log M\right),
	\end{multline*}when $M(\log M)^\ell\ll N=o\left( M\frac{\log M}{\log\log M}\right)$, $0<\ell<1$ small enough.
	
\end{proof}	
Next step is to estimate the variance of $\Psi_f(N,M)$.
\begin{prop}\label{prop4}
		Let $N,M>0$ such that$$M(\log M)^\ell\ll N=o\left( M\frac{\log M}{\log\log M}\right) $$for some $0<\ell<1$. Then$$
	\sigma^2_N(\Psi_f(N,M))\ll \frac{M^2(\log M)^2}{\log\log M}+NM\log M\log\log M,
	$$as $N,M\rightarrow+\infty$. 
\end{prop}
In particular, $\sigma^2_N(\Psi_f(N,M))=o(\mathbb{E}_N(	\Psi_f(N,M)))$ in the above range of $N,M$, so $\Psi_f(N,M)\sim\mathbb{E}_N(	\Psi_f(N,M))$ almost surely, which shows Theorem \ref{thm2}.
\begin{proof}
	One has $\sigma^2_N(\Psi_f(N,M))=\mathbb{E}_N(	\Psi_f(N,M)^2)-\mathbb{E}_N(	\Psi_f(N,M))^2$. The square of the mean value can easily be estimated by means of Proposition \ref{prop3}:\begin{multline*}
		\mathbb{E}_N(	\Psi_f(N,M))^2=(n-1)^2M^2(\log M)^2\\+O\left(\frac{M^2(\log M)^2}{\log\log M}+NM\log M\log\log M \right). 
	\end{multline*}It remains to study the average bahviour of $\Psi_f(N,M)^2$. We will be consistent with the notations of the previous proposition. Write$$	\Psi_f(N,M)=X_1+\dots+X_5,$$where the terms $X_i$ are defined at the beginning of the proof of Proposition \ref{prop3}, with the same order. Therefore
	\begin{align*}
			\Psi_f(N,M)&=X_1^2+\dots+X_5^2\\
			&=\pm 2\sum_{i\neq j=1,\dots,5}X_iX_j.
	\end{align*}

	\begin{enumerate}
		\item [$\bullet$]The first term is \begin{multline*}
			X_1^2=(\log P_f(M))^2=n^2M^2(\log M)^2\\+O\left( \frac{M^2(\log M)^2}{\log\log M}+NM\log M\log\log M\right)
		\end{multline*}if $N=o\left(\tfrac{M\log M}{\log\log M} \right) $.
		\item [$\bullet$] With the same argument of Proposition \ref{prop3} we can see that$$X_2^2\ll M^2.$$
		\item [$\bullet$]Write$$X_3^2=M^2\Big( \underset{\wp\tiny\mbox{ unr}}{\sum_{q_\wp\le M}}\frac{\log q_\wp}{q_\wp-1}s_\wp(f)\Big)^2+O(M^2\log M), $$where $ s_\wp(f) $ is the number of solutions of $f$ mod $ \wp $. We need an estimate on average of $\underset{\wp,\mathfrak{q}\tiny\mbox{ unr}}{\sum_{q_{\wp},q_{\mathfrak{q}}\le M}}s_\wp(f)s_{\mathfrak{q}}(f)$. If $M$ is small enough compared to $N$, we can use Proposition \ref{p9} to get\begin{align*}
			\mathbb{E}_N\Big(\underset{\wp,\mathfrak{q}\tiny\mbox{ unr}}{\sum_{q_{\wp},q_{\mathfrak{q}}\le M}}
			\underset{f(\alpha)\equiv0\ (\wp)}{\sum_{\alpha,\wp}}\ \underset{f(\beta)\equiv0\ (\mathfrak{q})}{\sum_{\beta,\mathfrak{q}}}1\Big)&=\sum_{\alpha,\beta}\underset{\tau\beta=\beta}{\underset{\sigma\alpha=\alpha}{\sum_{\sigma,\tau}}}\mathbb{E}_N\Big(\underset{\Frob_{f,\wp}=\sigma}{\sum_{q_{\mathfrak{q}}\le M\tiny\mbox{unr}}}\ \underset{\Frob_{f,\mathfrak{q}}=\sigma}{\sum_{q_{\mathfrak{q}}\le M\tiny\mbox{unr}}}1\Big)\\
			&=\pi_K(M)^2+O(\pi_K(M)\log\log M).
		\end{align*}We then have$$\mathbb{E}_N\Big( \underset{\wp,\mathfrak{q}\tiny\mbox{ unr}}{\sum_{q_{\wp},q_{\mathfrak{q}}\le M}}\sigma_\wp(f)\sigma_{\mathfrak{q}}(f)\Big)\ll\pi_K(M)\log\log M.$$
		By partial summation,\begin{multline*}
		\mathbb{E}_N\Big(\underset{\wp\tiny\mbox{ unr}}{\sum_{q_\wp\le M}}\frac{\log q_\wp}{q_\wp-1}s_\wp(f) \Big)=\underset{\wp\tiny\mbox{ unr}}{\sum_{q_\wp\le M}}\left( \frac{\log q_\wp}{q_\wp}\right) ^2\\+\mathbb{E}_N\Big(\Big(\underset{\wp\tiny\mbox{ unr}}{\sum_{q_\wp\le M}}\frac{\log q_\wp}{q_\wp}\sigma_\wp(f)\Big)^2\Big)\\
		=(\log M)^2+O((\log\log N)^2)+\mathbb{E}_N\Big(\Big(\underset{\wp\tiny\mbox{ unr}}{\sum_{q_\wp\le M}}\frac{\log q_\wp}{q_\wp}\sigma_\wp(f)\Big)^2\Big).
		\end{multline*}Fix $f$, and let $M'\le M$. The contribution of the "small primes" is given by\begin{multline*}
		\Big(\underset{\wp\tiny\mbox{ unr}}{\sum_{q_\wp\le M'}}\frac{\log q_\wp}{q_\wp}\sigma_\wp(f)\Big)^2=\Big(\underset{\wp\tiny\mbox{ unr}}{\sum_{q_\wp\le M'}}\sigma_\wp(f)\frac{\log M'}{M'}+\int_{2}^{M'}\sum_{q_\wp\le t}\sigma_\wp(f)\frac{1-\log t}{t^2}dt\Big)^2\\
		=\frac{(\log M')^2}{M'^{2}}\mathbb{E}_N\Big(\Big(\underset{\wp\tiny\mbox{ unr}}{\sum_{q_\wp\le M'}}\sigma_\wp(f)\Big)^2\Big)\\+\Big(\int_{2}^{M'}\sum_{q_\wp\le t}\sigma_\wp(f)\frac{1-\log t}{t^2}dt\Big)^2
		+O\left( \log\log M'\frac{\log M'}{M'}\right).
		\end{multline*}We use the Cauchy-Schwarz inequality to bound the square of the integral\begin{align*}
		\Big(\int_{2}^{M'}\sum_{q_\wp\le t}\sigma_\wp(f)\frac{1-\log t}{t^2}dt\Big)^2&\le\int_{2}^{M'}\Big(\underset{\wp\tiny\mbox{ unr}}{\sum_{q_\wp\le t}}\sigma_\wp(f)\Big)^2dt\cdot\int_{2}^{M'}\Big(\frac{1-\log t}{t^2}\Big)^2dt.
		\end{align*}On average we thus have\begin{align*}
		\mathbb{E}_N\Big(\Big(\underset{\wp\tiny\mbox{ unr}}{\sum_{q_\wp\le M'}}\frac{\log q_\wp}{q_\wp}\sigma_\wp(f)\Big)^2\Big)&\ll\log\log M'\frac{\log M'}{M'}+\int_{2}^{M'}\frac{t}{\log t}\log\log t dt\\
		&\ll M'^2.
		\end{align*}For the primes of norm $M'<q_\wp\le M$, simply use that$$\left( \sum_{M'<q_\wp\le M}\frac{\log q_\wp}{q_\wp}\right) ^2\ll 1$$since it is the tail of a convergent series as $M\rightarrow+\infty$.\\
		Pick now $M'=\tfrac{\log M}{\log\log M}<N^{d\xi/2n+1}$. It follows that$$\mathbb{E}_N(X_3^2)=M^2(\log M)^2+O\left( M^2(\log\log N)^2+M^2\log M+M^2\left( \frac{\log M}{\log\log M}\right) ^2\right) .$$
		\item [$\bullet$] $$X_4^2=\underset{\wp,\mathfrak{q}\tiny\mbox{ ramified}}{\sum_{q_{\wp},q_{\mathfrak{q}}\le M}}\alpha_\wp(M)\alpha_{\mathfrak{q}}(M)\log q_\wp \log q_{\mathfrak{q}}=\mbox{I}^2+\mbox{II}^2+2\mbox{I}\cdot\mbox{II}.$$In particular$$\mathbb{E}_N(\mbox{I}^2)\ll M^2\log\log N,$$and so\begin{align*}
			\mathbb{E}_N(\mbox{I}\cdot\mbox{II})&\ll M\log\log N\cdot\mathbb{E}_N(\mbox{II})\\
			&\ll M\log\log N\left( M+\frac{M^2}{N^{d\xi}}\log M\right)\\
			&\ll M^2\log\log N. 
		\end{align*}Finally,\begin{align*}
		\mathbb{E}_N(\mbox{II}^2)&=	\mathbb{E}_N\Big(\Big(\underset{\wp\tiny\mbox{ ram}}{\sum_{q_{\wp}\le M}}\log q_\wp\sum_{N_{K/\Q}\lambda\le M}\underset{f(\lambda)\equiv 0\tiny\ (\wp^k)}{\sum_{k\ge2}}1\Big)^2\Big)\\
		&\ll\frac{1}{N^{nd}}\sum_{q_{\wp},q_{\mathfrak{q}}\le M}\log q_\wp\log q_{\mathfrak{q}}\underset{N_{K/\Q}\lambda_i\le M}{\sum_{\lambda_1,\lambda_2}}\sum_{2\le k\ll\log N+\log M}\underset{f(\lambda_2)\equiv0\ (\mathfrak{q}^k)}{\underset{f(\lambda_1)\equiv0\ (\wp^k)}{\sum_{f}}}1\\
		&\ll M^2\sum_{q_{\wp},q_{\mathfrak{q}}\le M}\log q_\wp\log q_{\mathfrak{q}}\sum_{k\ge2}\left( \frac{1}{q_\wp^2q_{\mathfrak{q}}^2}\right)^k\\
		&\ll M^2, 
		\end{align*}which yields to$$\mathbb{E}_N(X_4^2)\ll M^2\log\log N.$$
		\item [$\bullet$]By using the exact same technique as in the previous result with the auxiliary function $G(\mu,\lambda)$, one gets$$\mathbb{E}_N(X_5^2)\ll M^2.$$
		\item [$\bullet$]Now for the cross products, we easily obtain estimates by using the ones we have for $X_1,\dots,X_5$. The only product giving a contribution is
\begin{multline*}
	\mathbb{E}_N(-2X_1X_3)=-2n M^2(\log M)^2\\+O\left( M^2\log\log N\log M+\frac{M^2(\log M)^2}{\log\log M}+NM\log M\log\log M\right).
\end{multline*}	
	All the other terms are negligible:
			\begin{align*}
			\mathbb{E}_N(X_1X_2)&\ll M^2\log M,\\
			\mathbb{E}_N(X_1X_4)&\ll M^2\log M\log\log N,\\
			\mathbb{E}_N(X_1X_5)&\ll M^2\log M,\\
			\mathbb{E}_N(X_2X_3)&\ll M^2\log M,\\
			\mathbb{E}_N(X_2X_4)&\ll M^2\log M\log\log N,\\
			\mathbb{E}_N(X_2X_5)&\ll M^2\log M,\\
			\mathbb{E}_N(X_3X_4)&\ll M^2\log \log N,\\
			\mathbb{E}_N(X_3X_5)&\ll M^2,\\
			\mathbb{E}_N(X_4X_5)&\ll M^2\log \log N.
		\end{align*}
	\end{enumerate}
	By putting everything together, we see that the terms of size $M^2(\log M)^2$ erase in $	\sigma^2_N(\Psi_f(N,M))$, which gives the desired upper bound.
	
\end{proof}

\pagebreak

\noindent
\footnotesize
\textbf{EPFL} \'{E}COLE POLYTECHNIQUE F\'{E}D\'{E}RALE DE LAUSANNE, INSTITUT DE MATH\'{E}MATIQUES\\
\textit{Email address}: \texttt{ilaria.viglino@yahoo.it}

\end{document}